\documentclass[11pt]{amsart}
\usepackage{amsmath,amssymb,amsthm}

\newcommand{\Delt}{\ensuremath{-2(D-2)^{-1}}}
\newcommand{\grad}{\operatorname{grad}}
\newcommand{\ela}{\ell_{\alpha}}
\newcommand{\elb}{\ell_{\beta}}

\newcommand{\mathH}{\mathbb{H}}
\newcommand{\mathR}{\mathbb{R}}
\newcommand{\mathZ}{\mathbb{Z}}

\newcommand{\bR}{\mathbf{R}}

\newcommand{\bg}{\mathbf{g}}

\newtheorem{definition}{Definition}

\newtheorem{proposition}[definition]{Proposition}
\newtheorem{theorem}[definition]{Theorem}
\newtheorem{corollary}[definition]{Corollary}
\newtheorem*{thrm}{Theorem}

\title{The vanishing rate of Weil-Petersson sectional curvatures}
\author{Scott A. Wolpert}

\begin{document}

\maketitle

\begin{abstract}
The Weil-Petersson metric for the moduli space of Riemann surfaces has negative sectional curvature.  Surfaces represented in the complement of a compact set in the moduli space have short geodesics. At such surfaces the Weil-Petersson metric is approximately a product metric. An almost product metric has sections with almost vanishing curvature.  We bound the sectional curvature away from zero in terms of the product of lengths of short geodesics on Riemann surfaces.  We give examples and an expectation for the actual vanishing rate.  
\end{abstract}

\section{Introduction.}

Let $\mathcal{T}$ be the Teichm\"{u}ller space of marked genus $g$, $n$-punctured Riemann surfaces $R$ with hyperbolic metrics.  Associated to the hyperbolic metrics on Riemann surfaces are the Weil-Petersson K\"{a}hler metric and geodesic-length functions on $\mathcal T$.  The metric is incomplete.  The metric completion, the augmented Teichm\"{u}ller space $\overline{\mathcal T}$, is $\operatorname{CAT(0)}$ - a simply connected complete metric space with non positive curvature, \cite{Wlcbms}.     

Convexity is a basic property of the geodesic-length, Weil-Petersson geometry. Convexity is a consideration for the large scale behavior of geodesics, for the Nielsen Realization Problem \cite{Kerck,Wlnielsen}, for the diameter of the moduli space \cite{CaPa}, for the classification of the action of the mapping class group \cite{DW2} and for the ergodicity of the geodesic flow \cite{BMW}.  The Weil-Petersson metric has negative sectional curvature \cite{Trmbook, Wlchern} with infimum negative infinity and supremum zero, except in the one-dimensional case where a section is necessarily holomorphic.  The holomorphic sectional curvature is bounded above by $-1/Area$, for the hyperbolic area of a surface. There are analyses of the uniform behavior of curvature depending on the systole and topological type; see the new work of Bridgeman and Wu on Ricci and scalar curvature \cite{BrWu}.   

Hyperbolic surfaces have a thick-thin decomposition consisting of regions where the injectivity radius is bounded below by a positive constant and a complementary thin region.  Thin regions are disjoint unions of collars and cusps. The Weil-Petersson metric, covariant derivative and curvature tensor exhibit an asymptotic product structure with a factor for each thick region that is not a three-holed sphere and a one complex dimensional factor for each collar.  A product metric has sectional curvature nulls - a $2$-plane has null curvature if its projection to each factor is at most one real dimensional.  In \cite[Corollary 22]{Wlcurv}, we show that the thick-thin decomposition characterizes the Weil-Petersson asymptotic flat subspaces.    

We now examine the rate of sectional curvatures tending to zero - the rate of strict convexity tending to convexity.  The Weil-Petersson sectional curvature involves the difference of three evaluations of the quartic form
\[
\int_{R}\int_{R}\alpha\bar \beta\, G\,\gamma\bar\delta\, dA\,dA,
\]
where $\alpha, \beta, \gamma$ and $\delta$ are harmonic Beltrami differentials and $G$ is the Green's function for $-2(D-2)^{-1}$, for $D$ the hyperbolic metric Laplace-Beltrami operator and $dA$ the hyperbolic area element.  The quartic form involves the mass distribution of the Beltrami differentials and the propagation decay of the Green's function.  The proof of negative sectional curvature \cite[Theorem 4.5]{Wlchern} involves two applications of H\"{o}lder's inequality. Bounding sectional curvatures from above involves bounding the difference of small quantities and analyzing almost H\"{o}lder equality.  

The Green's function $G$ can be given as a sum over paths in the universal cover.  We use Dehn's Theorem for parameterizing the families of simple paths crossing the thin collars.  For $\ela$ the length of a short geodesic $\alpha$, we show that the propagation decay of $G(p,q)$ is the product of $\ela^3$ for each thin collar and $\ela$ for each half thin collar crossed by the minimal path from $p$ to $q$.  The propagation decay is a basic consideration for almost vanishing curvatures.  

We combine our understandings for a general bound.
\begin{thrm}
	There is a positive constant $C^*$ depending on topological type, such that the Weil-Petersson sectional curvature is at most $-C^*\sigma^7$ for $\sigma$ the product of small geodesic-lengths.
\end{thrm} 
We discuss the expectation that the optimal exponent is three. In the final section we analyze the sectional curvature for three standard examples. The vanishing rates follow the expectations.      

\section{Preliminaries.}

A Riemann surface with a finite area hyperbolic metric is uniformized by a discrete group $\Gamma$ acting as isometries on the upper half plane $\mathH$.  A Riemann surface with hyperbolic metric can be considered as the union of a \emph{thick} region where the injectivity radius is bounded below by a positive constant and a complementary \emph{thin} region.   The totality of all thick regions of Riemann surfaces of a given topological type forms a compact set of metric spaces in the Gromov-Hausdorff topology.  A \emph{thin} is a disjoint union of collar and cusp regions.   

For a simple geodesic $\alpha$ of length $\ela$, the extended collar about the geodesic is the region  $\{d(p,\alpha)\le\tilde w(\alpha)\}$, for the distance $d(\ ,\ )$ and half width $\tilde w(\alpha)$, $\sinh \tilde w(\alpha)\sinh \ela/2=1$.  The half width is given as $\tilde w(\alpha)=\log 4/\ela + O(\ela^2)$ for $\ela$ small. An extended collar is covered by $\{d(z,i\mathR^+)\le\tilde w(\alpha)\}\subset\mathH$ with deck transformation generated by $z\rightarrow e^{\ela}z$.   The quotient of the extended collar by the cyclic group of deck transformations embeds into the Riemann surface.  For $z\in\mathH$ and $\alpha$ lifting to the imaginary axis, the region is approximately $\{\ela/2\le\arg z\le \pi-\ela/2\}$.  An extended cusp region is covered by the region $\{\Im z\ge 1/2\}$ with deck transformation $z\rightarrow z+1$.  The quotient of the cusp region by the cyclic group of deck transformations embeds into the Riemann surface.  To ensure that uniform bands around boundaries embed into the Riemann surface, we use \emph{collars} defined by covering regions $\{\ela\le \arg z\le \pi-\ela\}$ and \emph{cusp regions} defined by covering regions $\{\Im z\ge 1\}$. The width of a collar is $2w(\alpha)=2\log 2/\ela +O(\ela^2)$.   
\begin{proposition}\textup{\cite[Chapter 4]{Busbook}}
	For a Riemann surface of genus $g$ with $n$ punctures, given pairwise disjoint simple closed geodesics, the extended collars about the geodesics and the extended cusp regions are mutually disjoint.   
\end{proposition}
The systole $\rho$ of a Riemann surface is the length of the shortest closed geodesic.  The systole is twice the minimal value of the injectivity radius.  The injectivity radius is small on the thin region and satisfies $\operatorname{inj}(p)e^{\delta(p)}$ is bounded above and below by positive constants, where $\delta$ is the distance to the collar or cusp boundary.  The diameters of thick regions are bounded.  Consequently the diameter of the complement of the cusp regions is given as a sum of collar widths and a bounded remainder.  

A \emph{pair of pants decomposition} is the specification of $3g-3+n$ homotopically non trivial, disjoint simple closed curves decomposing a surface into subsurfaces of topological type $(0,3)$.

We use a mean value inequality.  For holomorphic $n$-differentials there are positive constants such that 
\begin{equation}\label{mvi}
|\varphi(ds^2)^{-n/2}|(p)\,\le\,C\int_{B(p,1)}|\varphi(ds^2)^{-n/2}|\,dA\,\le\,C'\operatorname{inj}(p)^{-1}\|\varphi\|_1,
\end{equation}
for the injectivity radius, $ds^2$ the hyperbolic metric, and $dA$ hyperbolic area.  The first inequality is established in $\mathH$ by applying the mean value inequality for holomorphic functions for metric circles about a point and integrating the radius.  The second inequality accounts for the covering multiplicity of balls. 

We also use a pointwise bound for the magnitude of an $n$-differential in a cusp region.  Represent a cusp by a neighborhood of the origin in the punctured unit disc.  The hyperbolic metric is $ds^2=(|dw|/|w|\log|w|)^2$.  A holomorphic $n$-differential $\varphi$, with $\varphi(ds^2)^{-n/2}$ bounded,  is bounded on the cusp region $|w|\le e^{-\pi}$ as follows
\begin{multline}\label{Sch}
|\varphi(ds^2)^{-n/2}|(w)\le\pi^{-n}e^{\pi}|w|(\log 1/|w|)^n\max_{|w|=e^{-\pi}}|\varphi(ds^2)^{-n/2}|\le\\ C'\pi^{-n}e^{\pi}|w|(\log 1/|w|)^n\|\varphi\|_1.
\end{multline}
To establish the two inequalities, the product $\varphi(ds^2)^{-n/2}$ is written as
$$f(w)(dw/w)^n((|w|\log 1/|w|)/|dw|)^n
$$ 
for $f(w)$ holomorphic and vanishing at the origin.  Apply the Schwarz Lemma for the disc $|w|\le e^{-\pi}$, to find the inequality $|f|\le e^{\pi}|w|\max_{|w|=e^{-\pi}}|f|$.  For the first inequality note that $|f|=\pi^{-n}|\varphi(ds^2)^{-n/2}|$ on $|w|=e^{-\pi}$.  For the second inequality note that on the cusp region boundary the injectivity radius is at least unity and apply (\ref{mvi}).  We will use (\ref{Sch}) to uniformly bound differentials in cusp regions.  We will apply the inequalities for a product $\phi\psi$ of quadratic differentials, noting that $|\phi\psi|=|\phi\overline{\psi}|$.   

Points of the Teichm\"{u}ller space $\mathcal T(\Gamma)$ are equivalence classes $\{(R,ds^2,f)\}$ of homotopy marked genus $g$, $n$-punctured Riemann surfaces with complete hyperbolic metrics and reference homeomorphisms $f:F\rightarrow R$ from a base surface $F$.  Triples are equivalent provided there is a conformal isomorphism homotopic to the composition of reference homeomorphisms.  Basic invariants of a hyperbolic metric are the lengths of the unique geodesic representatives of the non peripheral free homotopy classes.  A non peripheral free homotopy class $[\alpha]$ on $F$ determines the $\alpha$ geodesic-length function $\ela$ - the length of the representative geodesic on $R$.  Points of the moduli space of Riemann surfaces $\mathcal M(\Gamma)$ are conformal equivalence classes of Riemann surfaces with hyperbolic metrics. Teichm\"{u}ller space is an orbifold covering of the moduli space.  Mumford's compactness theorem provides that the subsets of $\mathcal M(\Gamma)$ with the systole bounded below by a positive constant are compact sets.

\section{The operator $\Delta=\Delt$ and its Green's function.}

The deformation equation for a hyperbolic metric involves the Laplace-Beltrami operator $D$ acting on $L^2(\mathH/\Gamma)$.   Solving for the deformed hyperbolic metric involves the operator $\Delta=\Delt$.  We recall the  properties of $D$ and $\Delta$.  The Laplace-Beltrami operator is essentially self-adjoint acting on $L^2(\mathH/\Gamma)$.  The integration by parts formula 
\[
\int_{\mathH/\Gamma}fDh\,dA  \,=\,-\int_{\mathH/\Gamma}\nabla f\nabla h
\,dA
\]
provides that the spectrum of $D$ is non positive and that $\Delt$ is a  bounded positive operator acting on $L^2(\mathH/\Gamma)$ with unit norm.  The maximum principle for the equation $(D-2)f=h$ provides that $2\max |f|\le \max |h|$, for $h$ continuous, vanishing at any cusps.  Assuming that $f$ vanishes at any cusps, at a maximum $p$ of $f$, then $Df(p)\le 0$ and consequently $f(p)\le -h(p)$; at a minimum $q$ of $f$ then $Df(q)\le 0$ and $2f(q)\le -h(q)$ (if $f(q)$ is negative the inequality for the absolute value follows).  
We specify the operator $\Delta=\Delt$ by a Green's function. Basic estimates show that $\Delta$ is a self-map of $C_0(\mathH/\Gamma)$.  By the above argument the operator has unit norm.   The inequalities also provide that $f$ is non negative if $h$ is non negative. We summarize the basic properties of the operator \cite{GT,Wells}. 

\begin{theorem}
	The operator $\Delt$ is self-adjoint, positive with unit norm on $L^2(\mathH/\Gamma)$ and positive with unit norm on $C_0(\mathH/\Gamma)$.  The operator has a positive symmetric integral kernel Green's function.
\end{theorem}  

The Green's function is given by the uniformization group sum
\[
G(p,q)=\sum_{\gamma\in\Gamma} -2Q_1(d(p,\gamma q)),
\]
for $Q_1$ an associated Legendre function and $d(\ ,\ )$ hyperbolic distance on $\mathH$, \cite[\S1., pgs. 147, 148 and 155]{Fay}.  The positive function $-Q_1$ has a logarithmic singularity at the origin and satisfies $-Q_1\approx e^{-2d(\ ,\ )}$ at large distance.  The fundamental solution $-2Q_1$ is the Green's function for the operator $\Delta$ acting on functions small at infinity on $\mathH$.

\section{Estimating exponential-distance sums.}

Basic quantities for the potential and deformation theory of hyperbolic surfaces are give as uniformization group sums.   The Green's function of the operator $\Delt$ has the form of a sum of the exponential hyperbolic distance $e^{-2d(\ ,\ )}$.  For a geodesic-length function $\ela$, the gradient $\grad \ela$, the Hessian $\operatorname{Hess}\ela$ and Riera's gradient pairing formula also have the $e^{-2d(\ ,\ )}$ sum form.  For the Green's function the sum is over the uniformization group $\Gamma$, while for $\grad\ela$ and $\operatorname{Hess}\ela$ the sum is over one-sided cosets and for Riera's formula the sum is over two-sided cosets.  A classical argument for estimating such sums starts with the observation that the summand satisfies a mean value bound.  Then the sum is naturally bounded by the integral over the $\Gamma$-orbit of a metric ball.  We show that the resulting estimate is not optimal for surfaces with small length geodesics.  

Our purpose is to give an efficient bound for the effect of small geodesic-lengths.  We show the effect on the group sum of $e^{-2d(p,q )}$ for separation of the  points by a collar of length $\ela$ is $\ela^3$ and separation of the points by a half collar is $\ela$.  We use the Dehn parameterization of isotopy classes to analyze the effect of crossing a collar on the lengths of paths between $p$ and the $\Gamma$-orbit of a point $q$. 

A \emph{multicurve} is a disjoint union of homotopically non trivial simple closed curves.  Dehn showed for a surface of topological type $(g,n)$ that multicurves modulo free homotopy are parameterized bijectively by $\mathZ^{6g-6+2n}$.   The parameterization is based on choosing a pair of pants decomposition, defining disjoint annular neighborhoods for each pants curve in the decomposition and  defining \emph{windows}, a closed interval on each boundary of an annular neighborhood.  A multicurve can be isotoped to a standard form so that it intersects each annular neighborhood efficiently (no embedded bigons) and only intersects boundaries of annular neighborhoods in windows.  In the complement of the annular neighborhoods, representatives are specified for the relative isotopy classes of arcs between windows. There are only a finite number of possible isotopy classes of such arcs.  In general a multicurve in standard form consists of representative arcs between windows, simple closed curves isotopic to pants curves and arcs inside annular neighborhoods connecting windows.  The third type arcs may wind (twist) positively or negatively around the annulus.  Dehn's coordinates are based on the number of arcs of each type in the standard form and the twisting numbers of arcs of the third type \cite[\S 1.2, Theorem 1.2.1]{HP}.  The core of the result is that the standard form for isotopy classes of multicurves count of arcs and twisting numbers are intrinsic parameters.  The result is based on isotopy to a standard form.  Dehn's approach extends to describing the isotopy classes of simple arcs between two fixed points.  We are only interested in connected simple arcs and so the parameters give an injection to $\mathZ^{6g-6+2n}$.   

We consider the upper-half plane $\mathH$ with hyperbolic distance $d(\ ,\ )$.  Consider $\Gamma$ the uniformization group of a surface of topological type $(g,n)$.  We are interested in the group sum
\[
K(p,q)=\sum_{\gamma\in\Gamma}e^{-2d(p,\gamma q)}
\] 
for $p$ and $q$ in $\mathH$.  Our approach for considering the sum is to consider the paths from $p$ to $q$ grouped into families.  The elements within a family are enumerated by varying the Dehn annular twist numbers over $\mathZ$.  

We give a lower bound for the exponential-distance sum in terms of the cube of the product of small geodesic-lengths. The width of a collar is $2\log 2/\ell + O(1)$ and so the square of a geodesic-length is the exponential of the negative collar width.  Similarly the square product of small geodesic-lengths provides a lower bound for the exponential negative of the surface diameter.

\begin{proposition}\label{Ksum}
	There is a positive constant $C''$ depending on topological type, such that for neither $p$ or $q$ in cusp regions
	\[
	K(p,q)\,\ge\,C'' \sigma^3,
	\]
	where $\sigma$ is the product of small geodesic-lengths.  
\end{proposition}    
\begin{proof}
	We begin by considering the twisting family of simple arcs crossing a single collar.  We give upper bounds for the lengths of the arcs crossing a collar $\mathcal C$ about a geodesic of length $\ela$.  The comparison arcs are given by concatenations of a minimal geodesic from one boundary of $\mathcal C$ to $\alpha$, circuits about $\alpha$ and a minimal geodesic from $\alpha$ to the second boundary of $\mathcal C$.  The exponential length sum is
	\begin{equation}\label{lengthsum}
	\sum_{n\in\mathZ}e^{-2(2w(\alpha)+|n|\ela)}.
	\end{equation}
Substituting the expansion $w(\alpha)=\log 1/\ela +O(1)$ for the collar half-width, the Riemann sum compares to and is bounded by the integral
\[
C\ela^3\int^{\infty}_0 e^{-2x}dx.
\]
In our overall argument a comparison arc will connect specified points on the boundaries of a collar. This endpoint condition can be satisfied by including a fraction of a circuit about $\alpha$.   The modification will increase the length of the comparison arc by at most $\ela$ and will decrease the bound by a factor of $e^{-\ela}$.  In summary the lower bound is $\ela^3$.

The kernel $K(p,q)$ is a sum of positive terms.  To estimate the kernel from below we only consider simple paths between $p$ and $q$ that cross each collar at most once.   First we adjust the size of the collars to ensure that a simple geodesic does not enter and leave a collar by crossing the same boundary.   For cusp regions of unit area, a simple geodesic cannot enter the cusp sub region of area one-half.  A sequence of collars with core lengths tending to zero converges in the compact-open topology for metric spaces to a pair of cusp regions.   It follows that for sufficiently small core lengths, simple geodesics do not enter and leave collars of area less than unity by crossing the same boundary.  This is the setting we consider for bounding the kernel.  We use these values for defining thick-thin decompositions. The number of and diameters of thick regions are appropriately bounded. For neither $p$ or $q$ in cusp regions, then a simple geodesic connecting $p$ and $q$ consists only of segments in thick regions, segments crossing collars and segments entering or leaving a collar if $p$ or $q$ lie in a collar.  We consider $\eta_{pq}$ the shortest geodesic connecting $p$ to $q$.  The geodesic crosses each collar at most once and consequently crosses thick regions a number of times bounded by the surface topology.  In the special case that $p$ and $q$ lie in a common collar, then the length of the segments of $\eta_{pq}$ in the collar is at most half the collar width.  Since the thick regions have bounded diameters, the length of $\eta_{pq}$ is bounded by the sum of the collar widths plus a constant depending only on the topology. 

There is a multi twisting family of arcs from $p$ to $q$ starting with the geodesic $\eta_{pq}$.  For each collar that $\eta_{pq}$ crosses there is a $\mathZ$-fold collection of twistings.  By Dehn's theorem each multi twisting represents a distinct isotopy class.  As above, the contribution to a lower bound for $K(p,q)$ for a collar about a geodesic is a factor of $C\ela^3$.  The overall estimate follows.
\end{proof}

As noted, the Green's function for the operator $\Delta=\Delt$ is given by the uniformization group sum
\[
G(p,q)=\sum_{\gamma\in\Gamma} -2Q_1(d(p,\gamma q)),
\]
for $Q_1$ an associated Legendre function and $d(\ ,\ )$ hyperbolic distance on $\mathH$.   The positive function $-Q_1$ has a logarithmic singularity at the origin and satisfies $-Q_1\approx e^{-2d(\ ,\ )}$ at large distance, \cite{Fay}.   

\begin{corollary}\label{Gsum}
	There is a positive constant depending on topological type, such that for neither $p$ or $q$ in cusp regions 
	\[
	G(p,q)\ge C''\sigma^3,
	\]
	where $\sigma$ is the product of small geodesic-lengths.  
\end{corollary}	
\begin{proof}
	Given the behavior of $-Q_1$ at zero and infinity, the function  is bounded below by a positive multiple of $e^{-2d(\ ,\ )}$.  The proposition provides the conclusion.  
	
\end{proof}

To bound the kernels $G$ or $K$ from above involves bounding the lengths of paths from below and accounting for the exponential growth of the number of geodesic paths.  We now observe that the above comparison arcs for the twisting family of simple arcs crossing a collar give an upper bound of the same magnitude $\ell^3$. Hyperbolic trigonometry provides the necessary bounds.  For a hyperbolic right triangle with side lengths $a, b$ and $c$, with $c$ opposite the right angle then $\cosh c =\cosh a \cosh b$. 

Noting that $\log \cosh x= x+O(1)$, we have the length relation $c=a+b+O(1)$. Consider in $\mathH$ the configuration of points $p$, $q$ on opposite sides of a complete geodesic $\alpha$ with $\hat p$, $\hat q$ the projections of the points to the geodesic.   Consider that the geodesic $\stackrel{\frown}{pq}$ from $p$ to $q$ intersects $\alpha$ at a point $r$. In the proof of the proposition we bounded the length $\stackrel{\frown}{pq}$ by the sum of lengths of $\stackrel{\frown}{p\hat p}$, $\stackrel{\frown}{\hat p\hat q}$ and $\stackrel{\frown}{\hat q q}$. The segment $\stackrel{\frown}{\hat p\hat q}$ covers the circuits about $\alpha$.  We refer to the right triangles $p\hat pr$ and $q\hat q r$ and consider $p,q$ on the boundary of a collar about $\alpha$.  Referring to the relation of triangle side lengths, we have that the length of $\stackrel{\frown}{pq}$ equals the $\alpha$ collar width plus the length of circuits plus a bounded remainder.  It follows that provided the circuit segment is longer than a positive threshold, then $\ell(\stackrel{\frown}{pq})$ is greater than the sum of the collar width and a positive multiple (less than unity) of the circuit length.  The resulting upper bound for the exponential length sum is
\[
\sum_{|n|\ell\ge c}e^{-2(w(\alpha)+c'|n|\ell)} .
\] 
Substituting the expansion for the collar half-width, the Riemann sum compares to and is bounded by the integral
\[
C\ela^3 \int_ce^{-2c'x}dx,  
\]
matching the lower bound for the sum.  

The present analysis of paths can be applied to understand the convexity of geodesic-length functions for surfaces with small geodesic-lengths.  For a geodesic $\alpha$ with lift $\tilde\alpha$ to $\mathbb{H}$, and corresponding deck transformation $A$ stabilizing $\tilde\alpha$, consider the exponential-distance coset sum
\[
P_{\alpha}(p)\,=\,\sum_{\gamma\in\langle A\rangle\backslash\Gamma}e^{-2d(\tilde\alpha,\gamma p)}.
\] 
The analysis of twisting families of arcs can be applied to determine the magnitude of $P_{\alpha}$ on components of the thick-thin decomposition.   In \cite[Theorem 3.11]{Wlbhv}, we showed that the Hessian of $\ela$ satisfies
\[
\langle\mu,P_{\alpha}\mu\rangle\,\le\,3\pi\operatorname{Hess}\ela[\mu,\mu]\,\le\,48\langle\mu,P_{\alpha}\mu\rangle,
\]   
for $\mu$ a harmonic Beltrami differential. Magnitude information for $P_{\alpha}$ combines with magnitude information for $\mu$ to provide bounds the Hessian for a surface with small geodesic-lengths.    

\section{Analyzing Weil-Petersson sectional curvature.}

We begin with the formulas of \cite[\S 4.]{Wlchern} for the Weil-Petersson metric and curvature tensor.   For the uniformization $\mathH/\Gamma$ of a finite hyperbolic area surface, the deformation holomorphic tangent space is the space of harmonic Beltrami differentials $B(\Gamma)$. The Weil-Petersson metric is $ds^2=2\sum g_{\alpha\bar\beta}dt_{\alpha}\overline{dt_{\beta}}$ for
\[g_{\alpha\bar\beta}\,=\,\langle\mu_{\alpha},\mu_{\beta}\rangle\,=\,\int_{\mathH/\Gamma}\mu_{\alpha}\overline{\mu_{\beta}}\,dA,
\]
for $\mu_{\alpha}, \mu_{\beta}\in B(\Gamma)$ and $dA$ the hyperbolic area element.   We will also write $\langle\ ,\ \rangle$ for the Hermitian product on complex functions.   The operator $\Delta$ is represented by the Green's function integral
\[
\Delta f(p)\,=\,\int_{\mathH/\Gamma}G(p,q)f(q)\,dA.
\]
To simplify notation for functions $f$ and $h$, we write
\[
(f,h)\,=\,\int_{\mathH/\Gamma} f(p)\Delta h(q)\,dA,
\]
where $(\ ,\ )$ is a complex bilinear pairing.  A product of Beltrami differentials $\mu\overline{\nu}$ is a function and the curvature considerations involve the products
\[
(\alpha\bar\beta,\gamma\bar\delta)\,=\,\langle\Delta(\mu_{\alpha}\overline{\mu_{\beta}}),(\mu_{\gamma}\overline{\mu_{\delta}})\rangle\,=\,\langle(\mu_{\alpha}\overline{\mu_{\beta}}),\Delta(\mu_{\gamma}\overline{\mu_{\delta}})\rangle.
\]
The Weil-Petersson Riemann tensor is
\[
R_{\alpha\bar\beta\gamma\bar\delta}\,=\,(\alpha\bar\beta,\gamma\bar\delta)\,+\,(\alpha\bar\delta,\gamma\bar\beta), 
\]
\cite[Theorem 4.2]{Wlchern}.

We review Bochner's description \cite[Formulas 24 and 25]{Boch} of sectional curvature and the considerations of \cite[Theorem 4.5]{Wlchern}.  Given holomorphic tangent vectors $\tau_1, \tau_2$, associate the real tangent vectors $v_1=\tau_1+\overline{\tau_2},v_2=\tau_2+\overline{\tau_2}\in \mathbf{T}_{\mathR}\mathcal T(\Gamma)$.  Bochner shows that the curvature of the section spanned by $v_1$ and $v_2$ is $\bR/\bg$ where
\[
\bR\,=\,R_{1\bar21\bar2}-R_{1\bar22\bar1}-R_{2\bar11\bar2}+R_{2\bar12\bar1}
\]
and
\[
\bg\,=\,4g_{1\bar1}g_{2\bar2}-2|g_{1\bar2}|^2\,-\,2\Re (g_{1\bar2})^2.
\]

We represent the tangent vectors $\tau_1,\tau_2$ by Beltrami differentials $\mu_1,\mu_2\in B(\Gamma)$.  Sectional curvature depends on the $2$-plane spanned by the pair of vectors.  We normalize that $\mu_1,\mu_2$ are orthonormal.  The denominator $\bg$ then equals $4$. 

Starting from the earlier considerations, we have that
\begin{equation}\label{Rform}
\bR\,=\,4\Re(1\bar2,1\bar2)-2(1\bar2,2\bar1)-2(1\bar1,2\bar2).
\end{equation} 
We proceed to provide lower bounds for combinations of terms of $-\bR$.  We begin by writing $\mu_1\overline{\mu_2}=f+ih$ for the decomposition of the product into real and imaginary parts and writing $f=f^+-f^-$ for the decomposition of the real part into positive and strictly negative parts. We first consider the difference
\[
(\mu_1\overline{\mu_2},\mu_2\overline{\mu_1})-\Re(\mu_1\overline{\mu_2},\mu_1\overline{\mu_2})\,=\,(f,f)+(h,h)-(f,f)+(h,h)\,=\,2(h,h)
\]
and note that $(h,h)$ is non negative by the spectral decomposition of $\Delta$.  (In fact $h$ is non trivial since $\mu_1$ and $\mu_2$ are linearly independent.)   Next we consider a H\"{o}lder inequality for the operator $\Delta$.  The kernel $G$ is positive and has a positive square root.  For bounded functions $u,v$, we can write $G|uv|=G^{1/2}|u|G^{1/2}|v|$ and apply the H\"{o}lder inequality to find
\[
\big|\int Guv\,dA\big|\,\le\,\int G|uv|\,dA\,\le\,\Big(\int Gu^2\,dA\Big)^{1/2}\Big(\int Gv^2\,dA\Big)^{1/2}.
\]
The overall analysis of $-\bR$ focuses on the following sequence of inequalities
\begin{multline}\label{mainineq}
\Re(\mu_1\overline{\mu_2},\mu_1\overline{\mu_2})\,\le\,(|\Re\mu_1\overline{\mu_2}|,|\Re\mu_1\overline{\mu_2}|)\,\le\,(|\Re\mu_1\overline{\mu_2}|,|\mu_1\overline{\mu_2}|)\,\le\\ \int|\Re\mu_1\overline{\mu_2}|(\Delta|\mu_1|^2)^{1/2}(\
\Delta|\mu_2|^2)^{1/2}\,dA\,\le\\ \int|\mu_1\overline{\mu_2}|(\Delta|\mu_1|^2)^{1/2}(\
\Delta|\mu_2|^2)^{1/2}\,dA\,\le \\
\Big(\int|\mu_1|^2\Delta|\mu_2|^2\,dA\Big)^{1/2}\Big(\int|\mu_2|^2\Delta|\mu_1|^2\,dA\Big)^{1/2},
\end{multline}
where the first and second inequalities follow since $\Delta$ is positive, the third inequality follows from the H\"{o}lder inequality for $\Delta$ and the fifth inequality is a second application of H\"{o}lder's inequality. The final quantity equals $(\mu_1\overline{\mu_1},\mu_2\overline{\mu_2})$ since $\Delta$ is self-adjoint.  Our analysis will be based on two instances of the observation that in a sequence of inequalities the difference of any two entries is a lower bound for the difference of the first and last entries.  The first and last entries of (\ref{mainineq}) are terms of (\ref{Rform}).  We are ready to present the main result.
\begin{theorem}\label{main}
	There is a positive constant $C^*$ depending on topological type, such that the Weil-Petersson sectional curvature is at most $-C^*\sigma^7$ for $\sigma$ the product of small geodesic-lengths.
\end{theorem}
\begin{proof} Sectional curvatures are negative for the moduli space of Riemann surfaces. By Mumford's compactness theorem it suffices to establish a bound for surfaces with sufficiently small geodesic-lengths. We begin by setting values for constants.  We  use inequality (\ref{Sch}) to  modify the definition of cusp regions, to provide a cusp-norm inequality for holomorphic quartic differentials $\int_{cusps}|\varphi(ds^2)^{-2}|\,dA\le 1/8\|\varphi\|_1$. We use this inequality when applying Corollary \ref{Gsum} to bound by integrals over the complement of cusp regions.  We use the modified cusp region definition throughout the following discussion.  
Let $C'$ be the constant for the mean value inequality (\ref{mvi}) for quartic differentials and $C''$ the constant for Corollary \ref{Gsum} for the modified cusp regions; let $C'''$ be a positive constant such that $(C'')^2>22\,C'C'''$.  For the inequalities (\ref{mainineq}), we show that if the difference of the first and fifth terms is smaller than $C'''\sigma^7$, then the difference of the first and last terms is at least a positive multiple of $\sigma^3$ - in consequence $-\bR$ is at least a positive multiple of $\sigma^7$, the desired conclusion. 

We work with functions truncated to the complement of the modified cusp regions; for a function $k$ we write $\widetilde k$ for its truncation.   
We begin for $\mu_1\overline{\mu_2}=f^+-f^-+ih$ with the difference of the first two terms of (\ref{mainineq}) 
\begin{multline*}
(|\Re\mu_1\overline{\mu_2}|,|\Re\mu_1\overline{\mu_2}|)-\Re(\mu_1\overline{\mu_2},\mu_1\overline{\mu_2})\,=\,\\(f^+,f^+)+2(f^+,f^-)+(f^-,f^-)-(f^+,f^+)+2(f^+,f^-)-(f^-,f^-)+(h,h)
\,=\\4(f^+,f^-)+(h,h)\,\le\,C'''\sigma^7,
\end{multline*}
with $(h,h)$ positive. Since $f^+, f^-$ and $G$ are positive, we have that  $(\widetilde{f^+},\widetilde{f^-})\le(f^+,f^-)$ and by Corollary \ref{Gsum} that 
\[
4C''\sigma^3\int \widetilde{f^+}\,dA\int \widetilde{f^-}\,dA\,\le\,4(\widetilde{f^+},\widetilde{f^-})\,<\,C'''\sigma^7.
\]
Now since $\mu_1$ and $\mu_2$ are orthogonal, $\Re\int\mu_1\overline{\mu_2}\,dA=0$ and consequently $\int f^+\,dA=\int f^-\,dA$.  Observing next that 
\[
\Big|2\int \widetilde{f^+}\,dA - \int |\widetilde{f}|\,dA\Big| \quad\mbox{and}\quad \Big|2\int \widetilde{f^-}\,dA-\int |\widetilde{f}|\,dA\Big|
\] 
are each bounded by $1/8\|\mu_1\overline{\mu_2}\|_1$, we 
write the resulting inequality as
\begin{equation}\label{intbound}
C''\Big(\int|\widetilde{f}|\,dA\,\pm\,\frac18\|\mu_1\overline{\mu_2}\|_1\Big)^2\,<\,C'''\sigma^4,
\end{equation}
where $\pm$ denotes adding or subtracting a term no larger than $1/8\|\mu_1\overline{\mu_2}\|_1$.     

Next we consider the difference of the fourth and fifth terms of (\ref{mainineq}); again we may truncate since terms are positive
\[
\int \big(|\widetilde{\mu_1\overline{\mu_2}}|-|\Re\widetilde{\mu_1\overline{\mu_2}}|\big)(\Delta|\widetilde{\mu_1}|^2)^{1/2}(\
\Delta|\widetilde{\mu_2}|^2)^{1/2}\,dA\,\le\,C'''\sigma^7.
\]
We apply Corollary \ref{Gsum} to observe that $\Delta|\widetilde{\mu_1}|^2$ and $\Delta|\widetilde{\mu_2}|^2$ are pointwise bounded below by $C''\sigma^3\int|\widetilde{\mu_1}|^2\,dA$ and $C''\sigma^3\int|\widetilde{\mu_2}|^2\,dA$.  Each differential has unit norm and so applying the cusp-norm inequality, the above inequality gives
\[
\frac78 C''\int |\widetilde{\mu_1\overline{\mu_2}}|-|\Re\widetilde{\mu_1\overline{\mu_2}}|\,dA\,\le\,C'''\sigma^4,
\]
which we rewrite as
\[
\frac78 \int |\widetilde{\mu_1\overline{\mu_2}}|\,dA\,\le\,\frac78 \int |\Re\widetilde{\mu_1\overline{\mu_2}}|\,dA + C'''/C''\sigma^4.  
\]
Using that $f=\Re\mu_1\overline{\mu_2}$, we substitute (\ref{intbound}), to conclude that 
\[
\frac78 \int|\widetilde{\mu_1\overline{\mu_2}}|\,dA\,\le\,(C'''/C'')^{1/2}\sigma^2  +  C'''/C''\sigma^4+ \frac18\|\mu_1\overline{\mu_2}\|_1. 
\]
We apply the cusp-norm inequality on the left, absorb the norm term onto the left and assume that $C'''/C''\sigma^4<1$; we have the inequality 
\begin{equation}\label{muineq}
\frac{41}{64}\,\|\mu_1\overline{\mu_2}\|_1\,\le\,2(C'''/C'')^{1/2}\sigma^2.
\end{equation}

We estimate $(|f|,|f|)$.  Combine the above inequality and the mean value inequality to find that on the complement of the original cusp regions 
$|\mu_1\overline{\mu_2}|\,\le\,4\,C'(C'''/C'')^{1/2}\rho^{-1}\sigma^2$, for $\rho$ the surface systole and $\sigma$ sufficiently small. Then by inequality (\ref{Sch}), we have that $|\mu_1\overline{\mu_2}|\,\le\,5\,C'(C'''/C'')^{1/2}\sigma^2$ in cusp regions.  Since $|f|\,\le\,|\mu_1\overline{\mu_2}|$ and $\Delta$ has unit norm acting on $C_0$, we have the pointwise bound $\Delta|f|\,\le\,4\,C'(C'''/C'')^{1/2}\sigma$.  Combining with the $L^1$ bound (\ref{muineq}), we conclude that $(|f|,|f|)\le 16\,C'C'''/C''\sigma^3$.  By comparison, by Corollary \ref{Gsum}, we have that $(|\mu_1|^2,|\mu_2|^2)\ge(|\widetilde{\mu_1}|^2,|\widetilde{\mu_2}|^2)\ge 49/64\, C''\sigma^3$.  By choice of the constant $C'''$, the difference of the second and last terms of (\ref{mainineq}) is bounded below by a positive multiple of $\sigma^3$.   The proof is complete.   	
	
\end{proof}

\section{Vanishing rates for three examples.}

We consider three basic examples.  The tangents are given by geodesic-length gradients for simple geodesics that either have small length or bounded length and are contained in a thick region.  Our estimates will only use the absolute values of Beltrami differentials.  The conclusions are valid for $2$-planes spanned by complex multiples of the specified geodesic-length gradients.   

The analysis will be based on the present Corollary \ref{Gsum}, as well as Proposition 6, Corollary 9, Theorem 17 and Corollary 18 of \cite{Wlcurv}. 
In the earlier work and in the examples we write $\mu_{\alpha}$ for the gradient $\grad \ela$.  In the expansion for $\mu_{\alpha}$, we represent the geodesic $\alpha$ as the imaginary axis in $\mathH$.  
Introducing polar coordinates, let $\sin_{\alpha}\theta$ be the restriction of $\sin \theta$ to the collar $\ela\le\theta\le \pi-\ela$ with $\sin_{\alpha}\theta$ vanishing on the collar complement.  For easy reference we gather the earlier results into a single statement.    

\begin{theorem}\label{priorest}
A geodesic-length gradient $\mu_{\alpha}$ satisfies a general bound
$$
|\mu_{\alpha}(p)| \ \mbox{ is }\ O(\operatorname{inj}(p)^{-1}\ela e^{-d(\alpha,p)}).
$$
Geodesic-length gradients $\mu_{\alpha},\mu_{\beta}$ have expansions
$$
|\mu_{\alpha}|\,=\,a_{\alpha}(\alpha)\sin_{\alpha}^2\theta+a_{\alpha}(\beta)\sin_{\beta}^2\theta + O(\ela^2)
$$
and for $\alpha$ and $\beta$ disjoint,
$$
\mu_{\alpha}\overline{\mu_{\beta}}\,=\,a_{\alpha}(\alpha)a_{\beta}(\alpha)\sin_{\alpha}^4\theta + a_{\alpha}(\beta)a_{\beta}(\beta)\sin_{\beta}^4\theta + O(\ela^2\elb^2),
$$
and
$$
2\Delta\mu_{\alpha}\overline{\mu_{\beta}}\,=\,a_{\alpha}(\alpha)a_{\beta}(\alpha)\sin_{\alpha}^2\theta + a_{\alpha}(\beta)a_{\beta}(\beta)\sin_{\beta}^2\theta + O(\ela^2\elb^2),
$$
where the real principal coefficients satisfy $a_{\alpha}(\alpha)=2/\pi +O(\ela^3)$ and $a_{\alpha}(\beta)$ is $O(\ela^2\elb)$.

\noindent The gradient $\mu_{\alpha}$ has the $\beta$ collar expansion
\begin{multline}\label{mucollexp}
|\mu_{\alpha}(z)|\,=\,
a_{\alpha}(\beta)\sin^2\theta\ + \\
O\Big((\ela/\elb)^2\big((\max_{\theta = \ela/2}|\mu_{\alpha}|)e^{-2\pi\theta/\ela}\,+\,(\max_{\theta=\pi-\ela/2}|\mu_{\alpha}|)e^{2\pi(\pi-\theta)/\ela}\big)\sin^2\theta\Big),
\end{multline}
where the maxima are for the individual collar boundaries and on a given collar boundary the contribution from the opposite boundary is exponentially small.  For $c_0$ positive, all remainder terms are uniform for $\ela,\elb\le c_0$.
\end{theorem}

An immediate consequence of the theorem is that for $\ela$ small, $\mu_{\alpha}$ is mainly supported in the $\alpha$ collar, while for $\beta$ in a thick region, then $\mu_{\beta}$ is mainly supported in that same region. 

The calculation of sectional curvature involves the pairing of vectors.  Since we are bounding sectional curvature away from zero, a concern is that the denominator of $\bR/\bg$ could be large. The norm of geodesic-length gradients is bounded provided the geodesic-lengths are bounded, a hypothesis for our examples.  As noted, if a geodesic-length $\ela$ is small, then the unit-normalized differential is $\nu_{\alpha}=(\pi/2\ela)^{1/2}\mu_{\alpha}$.  In the examples we use the normalized differentials $\nu_{\alpha}$, for $\ela$ small.  Our considerations for the examples show that the geodesic-length gradients are almost orthogonal for small length; this information is not needed since the concern is a large denominator for $\bR/\bg$. 

As a preliminary matter we consider the maximum of $\sin\theta e^{-c\theta/\ell}$ for $c$ positive.   Forming the derivative and equating to zero gives that the maximum occurs for $\tan\theta=\ell/c$, in particular for $\theta\approx \ell/c$.  It follows that the maximum is approximately a constant multiple of $\ell$.   

We are ready to consider the  examples.   

{\emph{A pair of adjacent thick regions.}  Consider gradients $\mu_{\beta_1}$, supported on one thick region and $\mu_{\beta_2}$ supported on a second thick region with the regions connected by the collar for a short geodesic $\alpha$.  By the Theorem \ref{priorest} general bound for gradients the product $\mu_{\beta_1}\overline{\mu_{\beta_2}}$ is $O(\ela^2)$ on the thick regions.  We consider the product on the collar.  The product of principal coefficients $a_{\beta_1}(\alpha)a_{\beta_2}(\alpha)$ is $O(\ela^2)$.  We consider the product of principal and remainder terms using (\ref{mucollexp}).  
The product has maximum occurring for $\theta\approx c\ela$ and so a principal term $a_{\beta}(\alpha)\sin^2\theta$ has magnitude $O(\ela^3)$ and the remainder is at most $O(1)$. Next we consider the product of remainder terms.  On each collar boundary one of $\mu_{\beta_1}, \mu_{\beta_2}$ has magnitude $O(\ela^2)$ by the general bound.  
It follows that the product of remainders is $O(\ela^2)$ on the collar.  In conclusion $\mu_{\beta_1}\overline{\mu_{\beta_2}}$ is $O(\ela^2)$ and since $\Delta$ has unit norm as an operator on $C_0$, we have that $(\mu_{\beta_1}\overline{\mu_{\beta_2}},\overline{\mu_{\beta_1}}\mu_{\beta_2})$ is $O(\ela^4)$.  By Corollary \ref{Gsum}, $(\mu_{\beta_1}\overline{\mu_{\beta_1}},\mu_{\beta_2}\overline{\mu_{\beta_2}})$ is bounded below by a positive multiple of $\ela^3$.   The sectional curvature is at most a negative multiple of $\ela^3$ from formula (\ref{Rform}).

\emph{A collar adjacent to a thick region.}  Consider gradients $\mu_{\alpha}$ for the collar and $\mu_{\beta}$ for the adjacent thick region.  Consider the product $\mu_{\alpha}\overline{\mu_{\beta}}$ with the product of principal coefficients $a_{\alpha}(\alpha)a_{\beta}(\alpha)$ having magnitude $O(\ela)$.  We refer to Theorem \ref{priorest} to analyze
$\mu_{\alpha}\overline{\mu_{\beta}}\Delta\overline{\mu_{\alpha}}\mu_{\beta}$.  The product is $O(\ela^4)$ on the thick region with principal term $a_{\alpha}(\alpha)^2a_{\beta}(\alpha)^2\sin_{\alpha}^6\theta$ in the collar.  In the collar $\mu_{\alpha}\overline{\mu_{\beta}}$ and
$\Delta\mu_{\alpha}\overline{\mu_{\beta}}$ each have remainder terms that are $O(\ela^2)$ and $a_{\alpha}(\alpha)a_{\beta}(\alpha)$ is $O(\ela)$; it follows that in the collar the remainder of the product is $O(\ela^3)$.  The integral over the collar of $\sin^6\theta$ is 
\[
\int_1^{e^{\ela}}\int_{\ela}^{\pi-\ela}\sin^6\theta \frac{dr}{r}\frac{d\theta}{\sin^2\theta}.
\]  
The integral is $O(\ela)$ and the principal term is $O(\ela^3)$.  As noted, for small lengths the normalization of $\mu_{\alpha}$ requires a factor $\ela^{-1/2}$, so the overall bound is that  $(\nu_{\alpha}\overline{\mu_{\beta}},\overline{\nu_{\alpha}}\mu_{\beta})$ is $O(\ela^2)$.  Now consider $|\mu_{\beta}|^2\Delta|\mu_{\alpha}|^2$.  On the thick region it is bounded as $O(\ela^2)$ which is the expected order for its leading term - Theorem \ref{priorest} is not sufficient for the analysis.  Instead we use the approach of Proposition \ref{Ksum} and Corollary \ref{Gsum} based on the shortest path from $\alpha$ to $\beta$.  The earlier analysis applies, only now the shortest path crosses half a collar, rather than an entire collar.  
In the sum (\ref{lengthsum}) replace the collar width $2w(\alpha)$ with the collar half width $w(\alpha)$.  The consequence is that the Green's function is bounded below by a positive multiple of $\ela$.  We have that $(\nu_{\alpha}\overline{\nu_{\alpha}},\mu_{\beta}\overline{\mu_{\beta}})$ is at least a positive multiple of $\ela$ and from above the remaining contributions are $O(\ela^2)$.  The sectional curvature is at most a negative multiple of $\ela$.  

\emph{A pair of collars adjacent to a thick region.}  Consider gradients $\mu_{\alpha_1},\mu_{\alpha_2}$ for a pair of collars adjacent to a thick region.   By symmetry it suffices to consider the $\alpha_1$ collar.  Consider the product $\mu_{\alpha_1}\overline{\mu_{\alpha_2}}$ with the product of principal coefficients $a_{\alpha_1}(\alpha_1)a_{\alpha_2}(\alpha_1)$ having magnitude $O(\ell_{\alpha_1}\ell_{\alpha_2}^2)$.  We refer to Theorem \ref{priorest} to analyze $\mu_{\alpha_1}\overline{\mu_{\alpha_2}}\Delta\overline{\mu_{\alpha_1}}\mu_{\alpha_2}$.  
The product is $O(\ell_{\alpha_1}^4\ell_{\alpha_2}^4)$ on the thick region with principal term $a_{\alpha_1}(\alpha_1)^2a_{\alpha_2}(\alpha_1)^2\sin_{\alpha_1}^6\theta$ in the $\alpha_1$ collar.  In the collar, $\mu_{\alpha_1}\overline{\mu_{\alpha_2}}$ and $\Delta\mu_{\alpha_1}\overline{\mu_{\alpha_2}}$ each have remainders that are $O(\ell_{\alpha_1}^2\ell_{\alpha_2}^2)$  and  $a_{\alpha_1}(\alpha_1)a_{\alpha_2}(\alpha_1)$ is $O(\ell_{\alpha_1}\ell_{\alpha_2}^2)$; it follows that in the collar the remainder of the product is $O(\ell_{\alpha_1}^3\ell_{\alpha_2}^3)$.  
Now the integral of $\sin^6\theta$ over the collar is $O(\ell_{\alpha_1})$, as above.  
So the product $\mu_{\alpha_1}\overline{\mu_{\alpha_2}}\Delta\overline{\mu_{\alpha_1}}\mu_{\alpha_2}$ is $O(\ell_{\alpha_1}^3\ell_{\alpha_2}^3)$ and the product of normalized differentials $\nu_{\alpha_1}\overline{\nu_{\alpha_2}}\Delta\overline{\nu_{\alpha_1}}\nu_{\alpha_2}$ is $O(\ell_{\alpha_1}^2\ell_{\alpha_2}^2)$.  Now consider $|\nu_{\alpha_1}|^2\Delta|\nu_{\alpha_2}|^2$ and again we use the approach of Proposition \ref{Ksum} and Corollary \ref{Gsum} based on the shortest path from $\alpha_1$ to $\alpha_2$.  
The shortest path crosses half the $\alpha_1$ collar and half the $\alpha_2$ collar.  In the sum (\ref{lengthsum}) again the collar width is replaced by the half width. The Green's function is bounded below by a positive multiple of $\ell_{\alpha_1}\ell_{\alpha_2}$.  We have that $(\nu_{\alpha_1}\overline{\nu_{\alpha_1}},\nu_{\alpha_2}\overline{\nu_{\alpha_2}})$ 
is at least a positive $\ell_{\alpha_1}\ell_{\alpha_2}$ and the remaining contributions are $O(\ell_{\alpha_1}^2\ell_{\alpha_2}^2)$.  The sectional curvature is at most a negative multiple of $\ell_{\alpha_1}\ell_{\alpha_2}$.   

Theorem \ref{priorest} provides that Beltrami differentials can be decomposed into components approximately supported on the components of the thick-thin decomposition.  Accordingly for a surface with small geodesic-lengths, a product of unit-norm Beltrami differentials can be pointwise small.  Theorem \ref{main} provides a converse - for almost vanishing curvature the quantities of (\ref{mainineq}) are almost equal.  The product $\mu_1\overline{\mu_2}$ is suitably small. The examples exhibit the behavior of the product of Beltrami differentials being pointwise small.    

The examples suggest a general behavior for almost vanishing curvature. Given an orthonormal basis for a tangent $2$-plane, choose two components of the thick-thin decomposition such that the $L^2$-mass of each Beltrami differential is substantial on the combined components. Then choose a new basis for the $2$-plane with each differential approximately supported on one of the two components.  We expect the sectional curvature to correspond to the propagation decay of the Green's function between the two components. Crossing a half $\alpha$ collar contributes a propagation factor $\ela$ and crossing a full $\alpha$ collar contributes a propagation factor $\ela^3$.  We expect that the optimal value of the exponent in Theorem \ref{main} is three.    

The propagation decay of the Green's function $G(p,q)$ depends on the collars and half collars crossed by a minimal length path from $p$ to $q$.  Since the diameter of the thick region is bounded, the minimal path length is the sum of collar and half collar widths plus a bounded quantity.  In Proposition \ref{Ksum}, Corollary \ref{Gsum} and Theorem \ref{main} we use the sum of all collar widths about short geodesics as an upper bound for the minimal path length.  In fact the collars crossed by a minimal path depends on the topology of the location of the collars.  For pinching the \emph{A cycles} of a standard homology basis then a minimal path crosses at most two half collars - the maximal configuration is for the endpoints on distinct geodesics.  In this case the three results are valid with the product of the two smallest lengths replacing the product of all short lengths.  In some contrast, for the thick regions arranged in a sequence as a \emph{string of pearls} then a minimal path may cross all the collars and the product of all small geodesic-lengths is needed.  An analysis of the proof of Theorem \ref{main} shows that the exponent seven is twice the expected exponent for a collar plus one for application of the mean value estimate.  If an endpoint lies on a short geodesic and a minimal path at most crosses that half collar, then the resulting exponent is three.   In particular for pinching the A cycles, the three results are given in terms of the product of the two smallest lengths raised to the third power.    


\providecommand\WlName[1]{#1}\providecommand\WpName[1]{#1}\providecommand\Wl{Wlf}\providecommand\Wp{Wlp}\def\cprime{$'$}

\end{document}